\documentclass[11pt,reqno]{amsart}

\usepackage{mathrsfs}
\usepackage{amsfonts}
\usepackage{amssymb}
\usepackage{amsthm}
\usepackage{amsmath}
\usepackage{latexsym}
\usepackage{tikz}
\usepackage{geometry}

\newtheorem{theorem}{Theorem}[section]
\newtheorem{proposition}[theorem]{Proposition}
\newtheorem{lemma}[theorem]{Lemma}
\newtheorem{corollary}[theorem]{Corollary}
\newtheorem{remark}[theorem]{Remark}

\makeatletter \@addtoreset{equation}{section} \makeatother

\newcommand{\beq}{\begin{equation}}
\newcommand{\eeq}{\end{equation}}

\newcommand{\Rmnum}[1]{\expandafter\@slowromancap\romannumeral #1@}

\begin{document}

\title[3D incompressible Oldroyd-B model]
{Global small solutions of 3D incompressible Oldroyd-B model without damping mechanism}
\author[Y. Zhu]{Yi Zhu}
\address{Department of Mathematics, East China University of Science and Technology, Shanghai 200237, People's Republic of China}
\email{\tt zhuyim@ecust.edu.cn}

\date{}

\begin{abstract}
In this paper, we prove the global existence of small smooth solutions to the three-dimensional incompressible Oldroyd-B model without damping on the stress tensor. The main difficulty is the lack of full dissipation in stress tensor. To overcome it, we construct some time-weighted energies based on the special coupled structure of system. Such type energies show the partial dissipation of stress tensor and the strongly full dissipation of velocity.
In the view of treating ``nonlinear term'' as a ``linear term'', we also apply this result to 3D  incompressible viscoelastic system with Hookean elasticity and then prove the global existence of small solutions without the physical assumption (div-curl structure) as previous works.
\end{abstract}

\maketitle

\section{introduction}

The Oldroyd-B model describes the motion of some viscoelastic flows, for example, the system coupling fluids and polymers. It presents a typical constitutive law which does not obey the Newtonian law (a linear relationship between stress and the gradient of velocity
in fluids). Such non-Newtonian property may arise from the memorability of some fluids. Formulations about viscoelastic flows of Oldroyd-B type are first introduced by Oldroyd \cite{Oldroyd} and are extensively discussed in \cite{BCAH}.

The 3D incompressible Oldroyd-B model  can be written as follows

\begin{equation}\label{eq:1.1}
\begin{cases}
u_t + u\cdot \nabla u -  \mu \Delta u + \nabla p =  \mu_1 \nabla \cdot \tau,  \quad\quad  (t, x) \in \mathbb{R}^+ \times \mathbb{R}^3,\\
\tau_t + u\cdot \nabla \tau  + a \tau + Q(\tau, \nabla u) =  \mu_2 D (u),\\
\nabla \cdot u = 0,
\end{cases}
\end{equation}
with initial data
\begin{equation}\nonumber
u(0,x) = u_0(x), \quad \tau(0,x) = \tau_0(x), \quad x \in \mathbb{R}^3.
\end{equation}
Here $u=(u_1,u_2, u_3)^{\top}$ denotes the velocity, $p$ is the scalar pressure of fluid. $\tau$ is the non-Newtonian part of stress tensor which can be seen as a symmetric matrix here. $D(u)$ is the symmetric part of $\nabla u$,
\begin{equation}\label{du}
D(u) = \frac{1}{2} \big( \nabla u + (\nabla u)^{\top} \big),
\end{equation}
and $ Q$ is a given bilinear form which can be chosen as
\begin{equation}\label{defnQ}
Q(\tau, \nabla u)= \tau \Omega(u) - \Omega(u) \tau + b(D(u) \tau + \tau D(u)),
\end{equation}
where $\Omega(u)$ is the skew-symmetric part of $\nabla u$, namely
\begin{equation}\label{omegau}
\Omega(u) = \frac{1}{2} \big( \nabla u - (\nabla u)^{\top} \big).
\end{equation}
The coefficients $\mu, a,  \mu_1, \mu_2$ are assumed to be non-negative constants. When $a=0$, the system becomes Oldroyd-B model without damping which is concerned in this paper. $b \in [-1, 1]$ is a parameter and if $b=0$, we call the system corotational case.

For a self-contained presentation, we shall give a  brief derivation of system \eqref{eq:1.1}. Following \cite{CM}, the differential form of momentum conservation for homogenous and incompressible fluid can be written as
\begin{equation}\nonumber
\partial_t u + u\cdot \nabla u  = \nabla \cdot \sigma,
\end{equation}
 with the associated incompressible condition $\nabla \cdot u = 0$. The stress tensor $\sigma$ is  usually  written as $$\sigma = - p I + \tau_{total}.$$
 $\tau_{total}$ contains viscosity and other stresses.  In classical elastic
solid, the stress tensor  depends on the deformation. And for classical
viscous fluid, the stress tensor  depends on the rate of  deformation.
When we concern with Oldroyd-B model, the constitutive law is selected by
\begin{equation}\nonumber
\tau_{total} + \lambda_1 \frac{\mathcal{D}\tau_{total}}{\mathcal{D}t}   = 2\eta \big ( D(u) + \lambda_2   \frac{\mathcal{D} D(u)}{\mathcal{D}t}  \big),
\end{equation}
where, for any tensor $f(x,t)$, we have
\begin{equation}\nonumber
\frac{\mathcal{D}f}{\mathcal{D}t} = \partial_t f+ u\cdot \nabla f  + Q(f, \nabla u).
\end{equation}
Here $\lambda_1$ denotes the relaxation time and $\lambda_2$ the retardation time with $0\leq \lambda_2 \leq \lambda_1$. We can decompose $\tau_{total}$ into two parts: Newtonian part and elastic part, i.e., $\tau_{total} = \mathcal{N} + \tau$. We know that  $\mathcal{N} = 2 \eta \lambda_2 /\lambda_1 D(u)$ and thus $\tau$ satisfies the second equation of system \eqref{eq:1.1}  with
$$ a = \frac{1}{\lambda_1}, \quad \mu_2 = \frac{2\eta}{\lambda_1}(1-\frac{\lambda_2}{\lambda_1}), \quad \mu_1 = 1. $$
For more detailed derivation, we refer to \cite{Oldroyd, BCAH, CM}.

As one of the most popular constitutive laws, Oldroyd-B model of viscoelastic fluids has attracted many attentions and lots of excellent works have been done. Naturally, for a partial differential system, one key consideration is to derive the local and global existence of solutions. Achieved by Guillop\'{e} and Saut \cite{GS, GS2}, the local strong solutions exist and are unique. They also show these solutions are global provided that the coupling parameter and  initial data are small enough.
Other than Hilbert spaces $H^s$ considered in \cite{GS, GS2}, the result under $L^s-L^r$ framework is studied in \cite{FGO}.
In corotational case ($b = 0$), based on the basic energy equality, Lions and Masmoudi \cite{LM} proved the global existence of weak solutions. However, the case $b \neq 0$ is still not clear by now.
The theories of local solutions and global small solutions in (or near) critical Besov spaces were first studied by Chemin and Masmoudi \cite{CM}. Some delicate blow-up criterions were also shown in \cite{CM}.
And in \cite{LMZ}, Lei, Masmoudi and Zhou improved the criterion.
For more global existence results in generalized spaces, we refer to \cite{CM2, ZFZ}.
Further more, global well-posedness with a class of large initial data was given by Fang and Zi \cite{FZ}.
We should point out here the above results for global smooth solutions always require $a > 0$ (namely the system with damping) at least for non-trivial initial data.

In the case $\mu = 0$ and the equation of $\tau$ contains viscous term $-\Delta \tau$,  Elgindi and Rousset \cite{ER} proved the global existence of smooth solutions with small initial data in 2D. When $Q = 0$, they also derived the similar result with general data.
The key idea of \cite{ER} is to study a new quantity $\Gamma = \omega - \mathcal{R}\tau$, where $ \omega = \text{curl} \;u$ and $ \mathcal{R} = \Delta^{-1}\text{curl div}.$ For the 3D case, the small initial data result was obtained by Elgindi and Liu \cite{EL}.
Recently, applying direct energy method based on the coupling structure of system,  an improvement without the damping term ($a = 0$) was given by the author \cite{Zhu}.

Besides, we would like to mention that global
regularity of solutions to 2D Oldroyd-B model with diffusive stress  was obtained by Constantin and Kliegl \cite{CK}. And some interesting results for related Oldroyd  type models of viscoelastic fluids can be found in \cite{LinLZ, LZ, LeiLZ, LeiLZ2, LinZ, CZ, QZ, ZF}. We shall review these results later in Section 2.

Now, let us give the main result of this paper. We focus on the Oldroyd-B model in the case $a = 0$ (without damping term). More precisely, we prove the following theorem.

\begin{theorem}\label{thm}
Let $\mu,  \mu_1, \mu_2 >0$ and $ a = 0$. Suppose that $\nabla \cdot u = 0, (\tau_0)_{ij} =( \tau_0)_{ji}$ and initial data $|\nabla|^{-1}u_0, |\nabla|^{-1} \tau_0 \in H^3(\mathbb{R}^3)$.  Then there exists a small constant
$\varepsilon$ such that system \eqref{eq:1.1} admits a unique global classical solution provided that
$$ \||\nabla|^{-1}u_0\|_{H^3} + \||\nabla|^{-1}\tau_0\|_{H^3} \leq \varepsilon, $$
where $|\nabla| = (-\Delta)^\frac{1}{2}$.
\end{theorem}

\begin{remark}
The assumption of initial data in negative order Sobolev space $\dot H^{-1}$ could be removed by considering fractional order time-weighted energies in the energy framework \eqref{energy2}. To best illustrate our idea and make paper neat, here we wouldn't like to touch fractional order energies.
\end{remark}

The key point in proving theorem \ref{thm} is to obtain $L^1$ estimate of $\|\nabla u(t, \cdot)\|_{L^\infty_x}$ in time as well as some higher order norms of $u$.
It helps preserve the regularity of solutions from initial data.
However, the lack of full dissipation in stress tensor $\tau$ brings the main difficulty.
We analyse the following linearized system first (without loss of generality, set $\mu = \mu_1 = \mu_2 = 1$)
\begin{equation}\nonumber
\begin{cases}
u_t - \Delta u = \mathbb{P} \nabla \cdot \tau, \\
\tau_t = D(u).
\end{cases}
\end{equation}
 Here, $\mathbb{P}$ is the projection operator used to deal with  pressure.
 Notice the fact that
 $$ \mathbb{P} \nabla \cdot D(u) = \frac{1}{2}\Delta u ,$$
we can decouple the linearized system and find that both $u$ and $\mathbb{P} \nabla \cdot \tau$ satisfy the following damped wave equation
\begin{equation*}\label{lin}
W_{tt} - \Delta W_t - \frac{1}{2}\Delta W = 0.
\end{equation*}
It seems that we can gain enough decay at least in the linearized system \eqref{lin}.
In fact, we only expect the partial dissipation in $\tau$, namely $\mathbb{P} \nabla \cdot \tau$.
Based on the special dissipative mechanism of system \eqref{eq:1.1} we set two type energies (see \eqref{energy1} and \eqref{energy2}).
To enclose the energy with only partial dissipation of $\tau$, we make full use of the structure of system (some cancelations on linear terms) and the fact that time derivative of $\nabla \cdot \tau$ is essentially quadratic terms. In addition, to deal with the wildest term $\mathbb{P} \nabla \cdot (u \cdot \nabla \tau)$, we introduce a proposition related to some $\big[\mathbb{P}\; \text{div} , u\cdot \nabla \big]$ type commutator:
\begin{equation}\nonumber
\mathbb{P} \nabla \cdot(u\cdot \nabla \tau) = \mathbb{P} (u\cdot \nabla \mathbb{P}\nabla \cdot \tau)+ \text{some terms containing}\; \nabla u.
\end{equation}
For more details, we refer to Proposition \ref{prop} and the estimate \eqref{eqM4}.
~\\~
\section{Application for the Hookean elastic materials}

In this section, we will apply Theorem \ref{thm} to the following
three dimensional incompressible viscoelastic system with Hookean elasticity:
\begin{equation}\label{Hk}
\begin{cases}
u_t+ u\cdot \nabla u - \Delta u + \nabla p = \nabla \cdot (F F^\top), \quad\quad  (t, x) \in \mathbb{R}^+ \times \mathbb{R}^3,\\
F_t + u \cdot \nabla F = \nabla u F,\\
\nabla \cdot u = 0,\\
F(0,x) = F_0(x), \quad u(0,x) = u_0(x).
\end{cases}
\end{equation}
Here $u=(u_1,u_2, u_3)^{\top}$ is the velocity, $p$ presents the scalar pressure of fluid and $F$ denotes the deformation tensor. We adopt the following notations
\begin{equation}\nonumber
[\nabla u]_{ij} = \partial_j u_i, \quad [\nabla \cdot F]_{i} = \sum_{j} \partial_j F_{ij}.
\end{equation}
For detailed physical background, we refer to \cite{Larson, LinLZ} and references therein.

We consider the case $F$ is near an identity matrix and let $U = F - I$.
Through the analysis of linearized system of $(u, U)$, we can derive decay estimate of $\nabla \cdot U$.
However, there is no much more information about $\nabla \times U$.
To overcome this problem, Lin, Liu and Zhang \cite{LinLZ} studied an auxiliary vector field with the physical assumption $\nabla \cdot U_0^\top = 0$. They proved the global existence of classical small solutions in 2D case (we refer to \cite{LZ, LeiLZ2} for different approaches).
For the 3D case, Lei, Liu and Zhou \cite{LeiLZ} found a curl structure which is physical and compatible with system \eqref{Hk}.
They see $\nabla \times U_0$ as a higher order term and then proved the results of global small solutions in both 2D and 3D.
We also refer to Chen and Zhang \cite{CZ} for a curl free structure of $F_0^{-1}- I$.
Initial-boundary value problem was done by Lin and Zhang \cite{LinZ}. Qian and Zhang \cite{QZ} generalized the results to compressible case. The result of critical $L^p$ framework was given by Zhang and Fang \cite{ZF}.

Through a simple analysis of system \eqref{Hk} we know that the nonlinear term $\nabla \cdot (F F^\top)$ may be the
most difficult.
The above excellent works make sufficient use of the physical assumption presents div-curl structure of deformation tensor.
Now, we try to treat ``nonlinear term'' $\nabla \cdot (F F^\top)$
as a ``linear term'' and give the following formulation.
By considering $(u, F F^\top)$, we have
\begin{equation}\nonumber
\begin{cases}
u_t+ u\cdot \nabla u - \Delta u + \nabla p = \nabla \cdot (F F^\top),\\
(FF^\top)_t + u \cdot \nabla (FF^\top) = \nabla u F F^\top + FF^\top (\nabla u)^\top ,\\
\nabla \cdot u = 0.
\end{cases}
\end{equation}
Denote $G = F F^\top-  I$ and notice \eqref{du}, \eqref{omegau}, we have
\begin{equation}\nonumber
\begin{split}
\nabla u G + G(\nabla u)^\top
&= (D(u) + \Omega(u)) G +G (D(u)-\Omega(u)), \\
&= \Omega(u) G  - G \Omega(u)+D(u)G + G D(u). \\
\end{split}
\end{equation}
Thus, the system of $(u,G)$ can be written as
\begin{equation}\label{eqG}
\begin{cases}
u_t+ u\cdot \nabla u - \Delta u + \nabla p = \nabla \cdot G,\\
G_t + u \cdot \nabla G + Q(G, \nabla u)=  2 D(u), \\
\nabla \cdot u = 0.
\end{cases}
\end{equation}
It just becomes the Oldroyd-B model \eqref{eq:1.1} introduced in Section 1.
For more relations of these models we refer to \cite{Larson, LinLZ}.
Hence, we have the following corollary.

\begin{corollary}\label{cor}
Suppose that $\nabla \cdot u_0 = 0$ and initial data  $|\nabla|^{-1} u_0, |\nabla|^{-1}(F_0 - I)  \in H^3(\mathbb{R}^3)$. Then there exists a small constant $\varepsilon$ such that system \eqref{Hk} admits a
unique global classical solution provided that
$$ \||\nabla|^{-1}u_0\|_{H^3} + \||\nabla|^{-1}(F_0 - I)\|_{H^3} \leq \varepsilon.$$
\end{corollary}

Notice here we do not need the assumption of div-curl structure on the initial data, this corollary can also generalize the results in \cite{CZ, LeiLZ} somehow.

\begin{remark}
To prove this corollary, we first notice that
\begin{equation}\nonumber
\begin{split}
&\||\nabla|^{-1}(F_0F_0^\top - I)\|_{H^3}\\
 \lesssim& \||\nabla|^{-1}\big((F_0 - I + I)(F_0 - I + I)^\top - I \big)\|_{H^3} \\
\lesssim&  \||\nabla|^{-1}(F_0 - I)(F_0 - I)^\top\|_{H^3} + \||\nabla|^{-1}(F_0 - I)\|_{H^3} \\
\lesssim& \|(F_0 - I)(F_0 - I)^\top\|_{L^\frac{6}{5}} +   \|(F_0 - I)(F_0 - I)^\top\|_{H^2}  +  \||\nabla|^{-1}(F_0 - I)\|_{H^3}\\
\lesssim& \|(F_0 - I)\|_{L^2}\|(F_0 - I)^\top\|_{L^3}+   \|(F_0 - I)\|_{H^2}^2  +  \||\nabla|^{-1}(F_0 - I)\|_{H^3}\\
\lesssim & \||\nabla|^{-1}(F_0 - I)\|_{H^3} + \|(F_0 - I)\|_{H^3}^2.
\end{split}
\end{equation}
Then, apply Theorem \ref{thm} to the reformulated system \eqref{eqG}, we will get the global regularity of velocity (see \eqref{energy1} and \eqref{energy2}). Now, it is easy to obtain the global regularity of $F$.
\end{remark}

\begin{remark}
In \cite{leizhen, leizhen2}, Lei proved the global existence of classical solutions to 2D Hookean incompressible viscoelastic model with smallness assumptions only on the size of initial strain tensor. Indeed, the non-singular matrix $F$ can be decomposed uniquely in the form $F=(I+V)R$, where $V$ stands the strain matrix and $R$ the rotation matrix. He showed the smallness of strain tensor $V_0$ ensures the global regularity of system. In the proof of Corollary \ref{cor}, we present the similar result in a different view. In proving the global existence of classical solutions to 3D Hookean incompressible viscoelastic model, we just need the smallness of $F_0F_0^\top - I$ rather than $F_0-I$. Obviously, the smallness of $V_0$ ensures the smallness of $F_0F_0^\top - I$.
\end{remark}

\section{Energy estimate}

\subsection{Preliminary}
Without loss of generality, set $\mu = \mu_1 = \mu_2 = 1$.
At the first of proof, we introduce the setting of energy. Based on our analysis in Section 1, we define some time-weighted energies for system \eqref{eq:1.1}. First, we give the basic energy as follows,
\begin{equation}\label{energy1}
\begin{split}
\mathcal{E}_0(t) =& \sup_{0 \leq t' \leq t} (\||\nabla|^{-1}u(t')\|_{H^3}^2 + \||\nabla|^{-1}\tau(t')\|_{H^3}^2) + \int_{0}^{t} \| u(t')\|_{H^3}^2 + \||\nabla|^{-1}\mathbb{P}\nabla \cdot \tau \|_{H^2}^2 \;dt',
\end{split}
\end{equation}
where $\mathbb{P} = \mathbb{I}- \Delta^{-1} \nabla \text{div}$ is the projection operator.
For any smooth divergence free vector $v$, we have $\mathbb{P} v = v$. And for a scalar function $\phi$, we know that $\mathbb{P} \nabla \phi = 0$.
The projection operator $\mathbb{P}$ is used to deal with pressure term $p$ satisfying
 $$\Delta p = \partial_i \partial_j (\tau_{ij}-u_i u_j).$$
Also, we define two time-weighted energies which imply the dissipative structure of system \eqref{eq:1.1},
\begin{equation}\label{energy2}
\begin{split}
\mathcal{E}_1(t) =&  \sup_{0 \leq t' \leq t} (1+t')\big(\|u(t')\|_{H^2}^2 + 2\||\nabla|^{-1}\mathbb{P}\nabla \cdot \tau(t')\|_{H^2}^2 \big) \\
&+ \int_{0}^{t} (1+t')\Big[\|\nabla u(t')\|_{H^2}^2 + \| \mathbb{P}\nabla \cdot \tau \|_{H^1}^2\Big] \;dt',\\
\mathcal{E}_2(t) =& \sup_{0 \leq t' \leq t} (1+t')^2\big(\|\nabla u(t')\|_{H^1}^2 + 2\| \mathbb{P}\nabla \cdot \tau(t')\|_{H^1}^2\big) \\
&+ \int_{0}^{t} (1+t')^2\Big[\|\nabla^2 u(t')\|_{H^1}^2 + \|\nabla \mathbb{P}\nabla \cdot \tau \|_{L^2}^2\Big] \;dt'.
\end{split}
\end{equation}
Using interpolation inequality, we easily know that
 \begin{equation}\nonumber
  \mathcal{E}_1(t) \lesssim \mathcal{E}_0^\frac{1}{2}(t)\mathcal{E}_2^\frac{1}{2}(t).
  \end{equation}
Hence, we only need to derive the estimates of $\mathcal{E}_0(t)$ and $\mathcal{E}_2(t)$.

Next, we introduce a useful proposition to deal with $\big[\mathbb{P}\; \text{div} , u\cdot \nabla \big]$ type commutators.

\begin{proposition}\label{prop}
For any smooth tensor $[\tau_{ij}]_{3 \times 3}$ and three dimensional  vector $u$, it always holds that
\begin{equation}\nonumber
\mathbb{P} \nabla \cdot(u\cdot \nabla \tau) = \mathbb{P} (u\cdot \nabla \mathbb{P}\nabla \cdot \tau)+ \mathbb{P} (\nabla u\cdot \nabla \tau)
- \mathbb{P}(\nabla u\cdot \nabla \Delta^{-1} \nabla \cdot \nabla \cdot \tau),
\end{equation}
where the $ith$ component of $\nabla u\cdot \nabla \tau$ is
\begin{equation}\nonumber
[\nabla u\cdot \nabla \tau]_i = \sum_{j}\partial_j u \cdot \nabla \tau_{ij},
\end{equation}
and also
\begin{equation}\nonumber
[\nabla u\cdot \nabla \Delta^{-1} \nabla \cdot \nabla \cdot \tau]_i = \partial_i u\cdot \nabla \Delta^{-1} \nabla \cdot \nabla \cdot \tau.
\end{equation}
\end{proposition}

\begin{proof}
Using direct computation we have
\begin{equation}\nonumber
[ \nabla \cdot(u\cdot \nabla \tau)]_i =  \sum_{j}\partial_j (u\cdot \nabla \tau_{ij})
=\sum_{j}(\partial_ju\cdot \nabla \tau_{ij}) + \sum_{j}(u\cdot \nabla \partial_j\tau_{ij}).
\end{equation}
Though the notations in proposition, we can write
\begin{equation}\label{eq:2.10}
\mathbb{P} \nabla \cdot(u\cdot \nabla \tau) = \mathbb{P} (\nabla u\cdot \nabla \tau) + \mathbb{P} (u\cdot \nabla \nabla \cdot \tau).
\end{equation}
Denote $\mathbb{P}^\perp = \Delta^{-1} \nabla \text{div}$, we now compute the second part of \eqref{eq:2.10} as follows
\begin{equation}\nonumber
\begin{split}
\mathbb{P} (u\cdot \nabla \nabla \cdot \tau) =& \mathbb{P} (u\cdot \nabla \mathbb{P}\nabla \cdot \tau)  + \mathbb{P}(u\cdot \nabla \mathbb{P}^\perp \nabla \cdot \tau)\\
=&\mathbb{P} (u\cdot \nabla \mathbb{P}\nabla \cdot \tau)  + \mathbb{P}(u\cdot \nabla \Delta^{-1} \nabla \nabla \cdot \nabla \cdot \tau)\\
=&\mathbb{P} (u\cdot \nabla \mathbb{P}\nabla \cdot \tau)  + \mathbb{P}\nabla(u\cdot \nabla \Delta^{-1} \nabla \cdot \nabla \cdot \tau) - \mathbb{P}(\nabla u\cdot \nabla \Delta^{-1} \nabla \cdot \nabla \cdot \tau)\\
=&\mathbb{P} (u\cdot \nabla \mathbb{P}\nabla \cdot \tau) - \mathbb{P}(\nabla u\cdot \nabla \Delta^{-1} \nabla \cdot \nabla \cdot \tau).
\end{split}
\end{equation}
Hence, we have
\begin{equation}\nonumber
\mathbb{P} \nabla \cdot(u\cdot \nabla \tau) = \mathbb{P} (\nabla u\cdot \nabla \tau) +
\mathbb{P} (u\cdot \nabla \mathbb{P}\nabla \cdot \tau) - \mathbb{P}(\nabla u\cdot \nabla \Delta^{-1} \nabla \cdot \nabla \cdot \tau).
\end{equation}
\end{proof}

\subsection{\textit{A priori} estimate}
In this subsection, we shall derive the \textit{a priori} estimate of $\mathcal{E}_0(t)$ and $\mathcal{E}_2(t)$ respectively. First, we consider the basic energy $\mathcal{E}_0(t)$ and give the following lemma.

\begin{lemma}\label{lem1}
The energies are defined in \eqref{energy1} and \eqref{energy2}, then we have
\begin{equation}\nonumber
\mathcal{E}_0(t) \lesssim \mathcal{E}_0(0) + \mathcal{E}_0^\frac{3}{2}(t) + \mathcal{E}_2^\frac{3}{2}(t).
\end{equation}
\end{lemma}

\begin{proof}
We divide the proof of lemma into two parts. Define $\mathcal{E}_{0,1}$ and $\mathcal{E}_{0,2}$ as follows
\begin{align*}
\mathcal{E}_{0,1}(t) = &\sup_{0 \leq t' \leq t} (\||\nabla|^{-1}u(t')\|_{H^3}^2 + \||\nabla|^{-1}\tau(t')\|_{H^3}^2) + \int_{0}^{t} \| u(t')\|_{H^3}^2 \;d t', \\
\mathcal{E}_{0,2} (t) = &\int_{0}^{t}  \||\nabla|^{-1}\mathbb{P}\nabla \cdot \tau \|_{H^2}^2 \;dt' .
\end{align*}
Then $\mathcal{E}_0 = \mathcal{E}_{0,1} + \mathcal{E}_{0,2}$, we shall first give the estimate of $\mathcal{E}_{0,1}$ .
~\\~\\~
\textbf{First Step:}
~\\~\\~
Applying the operator $\nabla^k |\nabla|^{-1}(k = 0, \cdots, 3)$ to system \eqref{eq:1.1}. Then, taking inner product with
$\nabla^k |\nabla|^{-1} u $ for the first equation  and taking inner product with $\nabla^k |\nabla|^{-1} \tau$ for the second equation. Adding them up we have
\begin{equation}\label{eq:2.2}
\frac{1}{2} \frac{d}{dt}(\||\nabla|^{-1} u\|_{H^3}^2 + \||\nabla|^{-1}\tau\|_{H^3}^2) + \| u\|_{H^3}^2 = N_1 + N_2 + N_3 + N_4,
\end{equation}
where
\begin{equation}\nonumber
\begin{split}
N_1 =& \sum_{k = 0}^{3} \int \Big( \nabla^k |\nabla|^{-1} \nabla \cdot \tau \nabla^k |\nabla|^{-1}u + \nabla^k |\nabla|^{-1} D(u) \nabla^k |\nabla|^{-1} \tau \Big) \;dx,\\
N_2 =& -\sum_{k = 0}^{3} \int  \nabla^k |\nabla|^{-1} (u\cdot \nabla u) \nabla^k|\nabla|^{-1} u \;dx,\\
N_3 =& -\int  |\nabla|^{-1}(u \cdot \nabla \tau) |\nabla|^{-1} \tau \;dx
 -\sum_{k = 0}^{2} \int \nabla^k(u \cdot \nabla \tau) \nabla^k \tau\; dx,\\
N_4 =& -\sum_{k = 0}^{3} \int \nabla^k|\nabla|^{-1} Q(\tau, \nabla u) \nabla^k |\nabla|^{-1}\tau \;dx .
\end{split}
\end{equation}

For the first term $N_1$, using integration by parts and the symmetry $\tau_{ij} = \tau_{ji}$ we have
\begin{equation}\label{eqN1}
\begin{split}
N_1 =& \sum_{k = 0}^{3} \int  \Big( \nabla^k |\nabla|^{-1} \nabla \cdot \tau \nabla^k |\nabla|^{-1} u + \sum_{i,j=1}^3 \nabla^k |\nabla|^{-1} \frac{\partial_j u_i + \partial_i u_j}{2}\nabla^k |\nabla|^{-1}\tau_{ij} \Big) dx\\
=& \sum_{k = 0}^{3} \int \nabla^k |\nabla|^{-1}\nabla \cdot \tau \nabla^k |\nabla|^{-1} u \; dx  \\
& - \sum_{k = 0}^{3}\sum_{i,j=1}^3  \frac{\nabla^k |\nabla|^{-1} u_i \nabla^k |\nabla|^{-1}\partial_j\tau_{ij} + \nabla^k |\nabla|^{-1} u_j\nabla^k |\nabla|^{-1} \partial_i \tau_{ij}}{2} \;dx\\
=& \sum_{k = 0}^{3} \int \Big( \nabla^k |\nabla|^{-1} \nabla \cdot \tau \nabla^k  |\nabla|^{-1} u - \sum_{i,j=1}^3 \nabla^k |\nabla|^{-1} u_i \nabla^k |\nabla|^{-1}\partial_j\tau_{ij} \Big) dx\\
=& 0.
\end{split}
\end{equation}

For the second term $N_2$, by divergence free condition, H\"{o}lder inequality and Sobolev imbedding theorem, we directly know that
\begin{equation}\nonumber
\begin{split}
N_2 \lesssim& \|u\otimes u\|_{H^3} \||\nabla|^{-1}u\|_{H^3}\\
\lesssim& \|u\|_{L^\infty}\|u\|_{H^3}\||\nabla|^{-1}u\|_{H^3}\\
\lesssim& \|u\|_{H^3}^2\||\nabla|^{-1}u\|_{H^3}.
\end{split}
\end{equation}
Hence,
\begin{equation}\label{eqN2}
\begin{split}
\int_{0}^{t} |N_2(t')| dt' \lesssim & \sup_{0 \leq t' \leq t}\||\nabla|^{-1}u(t')\|_{H^3} \int_{0}^{t} \| u\|_{H^3}^2 dt'\\
\lesssim & \mathcal{E}_0^\frac{3}{2}(t).
\end{split}
\end{equation}

Similarly, for the next term $N_3$, notice the divergence free condition $\nabla \cdot u = 0$ we get
\begin{equation}\nonumber
\begin{split}
|N_3| \lesssim& \|u \otimes \tau \|_{L^2} \||\nabla|^{-1} \tau\|_{L^2} +\sum_{k = 1}^{2}\Big| \int \nabla^{k-1}(\nabla u \cdot \nabla \tau) \nabla^k \tau \; dx \Big|\\
\lesssim& \|u\|_{L^\infty} \||\nabla|^{-1} \tau\|_{H^1}^2 +\big(\|\nabla u\|_{L^\infty}\|\nabla \tau\|_{H^1}+ \|\nabla^2 u\|_{L^6}\|\nabla \tau\|_{L^3} \big)\|\nabla \tau\|_{H^1}\\
\lesssim& \big(\|\nabla u\|_{L^2}^\frac{1}{2}\|\nabla^2 u\|_{L^2}^\frac{1}{2}+\|\nabla^2 u\|_{H^1} \big)\||\nabla|^{-1} \tau\|_{H^3}^2.
\end{split}
\end{equation}
Thus, we have the following estimate
\begin{equation}\label{eqN3}
\begin{split}
\int_{0}^{t}|N_3(t')| \lesssim& \sup_{0 \leq t' \leq t} \||\nabla|^{-1} \tau(t')\|_{H^3}^2 \cdot \Big\{ \int_{0}^{t}(1+t')^{-\frac{3}{4}} (1+t')^\frac{1}{4}\|\nabla u\|_{L^2}^\frac{1}{2}(1+t')^\frac{1}{2}\|\nabla^2 u\|_{L^2}^\frac{1}{2} dt'\\
&+ \int_{0}^{t}\|\nabla^2 u\|_{H^1} dt' \Big\}\\
\lesssim & \mathcal{E}_0(t)\cdot \big(\mathcal{E}_1^\frac{1}{4}(t) \mathcal{E}_2^\frac{1}{4}(t) + \mathcal{E}_2^\frac{1}{2}(t)\big)\\
\lesssim & \mathcal{E}_0(t)\cdot \big(\mathcal{E}_0^\frac{1}{8}(t) \mathcal{E}_2^\frac{3}{8}(t) + \mathcal{E}_2^\frac{1}{2}(t)\big)\\
\lesssim &\mathcal{ E}_0(t)\cdot \big(\mathcal{E}_0^\frac{1}{2}(t) + \mathcal{E}_2^\frac{1}{2}(t)\big)\\
\lesssim &\mathcal{E}_0^\frac{3}{2}(t) + \mathcal{E}_2^\frac{3}{2}(t).
\end{split}
\end{equation}

Now, we turn to the last term $N_4$. Using H\"{o}lder inequality and Sobolev imbedding theorem, it yields
\begin{equation}\nonumber
\begin{split}
|N_4| \lesssim& \||\nabla|^{-1} Q\|_{L^2}\||\nabla|^{-1} \tau\|_{L^2} + \|Q\|_{H^2}\|\tau\|_{H^2}\\
\lesssim& \|Q\|_{L^\frac{6}{5}} \||\nabla|^{-1} \tau\|_{L^2}+ \big(\|\nabla u\|_{L^\infty}\|\tau\|_{H^2} + \|\nabla^2 u\|_{L^6} \|\tau\|_{W^{1,3}} + \|\nabla^3 u\|_{L^2} \|\tau\|_{L^\infty} \big )\|\tau\|_{H^2}\\
\lesssim& \|\nabla u\|_{L^3}\|\tau\|_{L^2}\||\nabla|^{-1} \tau\|_{L^2}+\|\nabla^2 u\|_{H^1} \|\tau\|_{H^2}^2\\
\lesssim& (\|\nabla u\|_{L^2}^{\frac{1}{2}} \|\nabla^2 u\|_{L^2}^{\frac{1}{2}} +  \|\nabla^2 u\|_{H^1} )\||\nabla|^{-1}\tau\|_{H^3}^2,
\end{split}
\end{equation}
which implies
\begin{equation}\label{eqN4}
\begin{split}
\int_{0}^{t} |N_4| \lesssim& \sup_{0 \leq t' \leq t} \||\nabla|^{-1}\tau(t')\|_{H^3}^2 \int_{0}^{t} \|\nabla u\|_{L^2}^{\frac{1}{2}} \|\nabla^2 u\|_{L^2}^{\frac{1}{2}} + \|\nabla^2 u\|_{H^1} \; dt'\\
\lesssim& \mathcal{E}_0(t)(\mathcal{E}_1^\frac{1}{4}(t) \mathcal{E}_2^\frac{1}{4}(t)+ \mathcal{E}_2^\frac{1}{2}(t))\\
\lesssim& \mathcal{E}_0^\frac{3}{2}(t) + \mathcal{E}_2^\frac{3}{2}(t).
\end{split}
\end{equation}

Integrating \eqref{eq:2.2} with time, then according to the estimates of $N_1 \sim N_4$, i.e., \eqref{eqN1}, \eqref{eqN2}, \eqref{eqN3}, \eqref{eqN4}, it holds that
\begin{equation}\label{eqE01}
\mathcal{E}_{0,1} (t)
\lesssim  \; \mathcal{E}_0(0) + \mathcal{E}_0^\frac{3}{2}(t) + \mathcal{E}_2^\frac{3}{2}(t).
\end{equation}
\textbf{Second Step:}
~\\~\\~
Next, we shall deal with the left part $\mathcal{E}_{0,2} (t)$.
Operating $\mathbb{P}$ on the first equation of system \eqref{eq:1.1}. Recall that
$\mathbb{P} = \mathbb{I}- \Delta^{-1}\nabla \text{div}$, we have
\begin{equation}\nonumber
u_t + \mathbb{P}(u \cdot \nabla u) - \Delta u = \mathbb{P} \nabla \cdot \tau.
\end{equation}
Applying operator $\nabla^k |\nabla|^{-1} (k = 0\sim 2 )$  to the above equation, then taking inner product with $\nabla^k |\nabla|^{-1}\mathbb{P} \nabla \cdot \tau$, we get
\begin{equation}\label{eq:2.3}
\||\nabla|^{-1}\mathbb{P}\nabla \cdot \tau \|_{H^2}^2 = N_5 + N_6 + N_7,
\end{equation}
where,
\begin{equation}\nonumber
\begin{split}
N_5 =& \sum_{k = 0}^2 \int -\nabla^k |\nabla|^{-1}\Delta u \nabla^k |\nabla|^{-1}\mathbb{P} \nabla \cdot \tau \;dx,\\
N_6 =& \sum_{k = 0}^2 \int \nabla^k |\nabla|^{-1}\mathbb{P}(u \cdot \nabla u) \nabla^k |\nabla|^{-1}\mathbb{P} \nabla \cdot \tau \;dx,\\
N_7 =& \sum_{k = 0}^2 \int \nabla^k |\nabla|^{-1}u_t \nabla^k |\nabla|^{-1}\mathbb{P} \nabla \cdot \tau \;dx.
\end{split}
\end{equation}

For the first term $N_5$, using H\"{o}lder inequality we have
\begin{equation}\nonumber
N_5
\lesssim \|\nabla u\|_{H^3} \| |\nabla|^{-1}\mathbb{P} \nabla \cdot \tau\|_{H^2}.
\end{equation}
Hence, we can obtain
\begin{equation}\label{eqN5}
\begin{split}
\int_{0}^t |N_5(t')| dt' \lesssim& \int_{0}^ t \|\nabla u\|_{H^3} \||\nabla|^{-1} \mathbb{P} \nabla \cdot \tau\|_{H^2} \; dt'\\
\lesssim& \mathcal{E}_{0,1}^\frac{1}{2}(t)\mathcal{E}_{0,2}^\frac{1}{2}(t).
\end{split}
\end{equation}

In the estimate of $N_6$, we will use the property that Riesz operator $\mathcal{R}_i = (-\Delta)^{-\frac{1}{2}} \nabla_i$ is $L^2$ bounded. Hence, it holds for any vector $v$, $ \|\mathbb{P}v \|_{L^2} \lesssim \|v \|_{L^2}$. Using H\"{o}lder inequality and Sobolev imbedding theorem we get
\begin{equation}\nonumber
\begin{split}
|N_6| \lesssim& \|u \otimes u\|_{H^2} \||\nabla|^{-1}\mathbb{P}\nabla \cdot \tau\|_{H^2},\\
\lesssim& \|u\|_{L^\infty} \|u\|_{H^2} \||\nabla|^{-1}\mathbb{P}\nabla \cdot \tau\|_{H^2},\\
\lesssim& \|u\|_{H^2} ^2 \||\nabla|^{-1}\mathbb{P}\nabla \cdot \tau\|_{H^2}.
\end{split}
\end{equation}
Obviously, we can derive the estimate of $N_6$ as follows
\begin{equation}\label{eqN6}
\begin{split}
\int_{0}^t |N_6(t')| dt' \lesssim & \sup_{0 \leq t' \leq t} \|u(t')\|_{H^2}\int_{0}^t \|u\|_{H^2} \||\nabla|^{-1}\mathbb{P}\nabla \cdot \tau\|_{H^2 } dt'\\
\lesssim& \mathcal{E}_0^\frac{3}{2}(t).
\end{split}
\end{equation}

Now, we turn to the last term $N_7$. Using integration by parts and the fact $\mathbb{P} u = u$, we rewrite $N_7$ into two parts,
\begin{equation}\label{eq:2.1}
\begin{split}
N_7 =& \sum_{k = 0}^2 \int \nabla^k |\nabla|^{-1}u_t \nabla^k  |\nabla|^{-1}\nabla \cdot \tau \;dx\\
=&\sum_{k = 0}^2 \frac{d}{dt} \int \nabla^k |\nabla|^{-1}u \nabla^k |\nabla|^{-1} \nabla \cdot \tau \;dx- \sum_{k = 0}^2 \int \nabla^k |\nabla|^{-1}u \nabla^k |\nabla|^{-1} \nabla \cdot \tau_t \;dx.
\end{split}
\end{equation}
According to the second equation of system \eqref{eq:1.1}, we have the following equality
\begin{equation}\label{eqq}
\nabla \cdot \tau_t + \nabla \cdot (u\cdot \nabla \tau) + \nabla \cdot Q(\tau, \nabla u) = \frac{1}{2}\Delta u.
\end{equation}
Applying this equality to the last part in \eqref{eq:2.1}, we shall get
\begin{align*}
&\sum_{k = 0}^2 \int \nabla^k |\nabla|^{-1}u \nabla^k |\nabla|^{-1} \nabla \cdot \tau_t \;dx \\
= &\sum_{k = 0}^2 \int \nabla^k |\nabla|^{-1} u \nabla^k |\nabla|^{-1}\Big[ \frac{1}{2}\Delta u -  \nabla \cdot (u\cdot \nabla \tau) - \nabla \cdot Q(\tau, \nabla u) \Big]dx.
\end{align*}
Then, applying H\"{o}lder inequality and Sobolev imbedding theorem, we obtain
\begin{equation}\nonumber
\begin{split}
&\Big|\sum_{k = 0}^2 \int \nabla^k |\nabla|^{-1}u \nabla^k  |\nabla|^{-1}\nabla \cdot \tau_t \;dx \Big|\\
\lesssim& \|u\|_{H^2}^2 + \|u\|_{H^2}\|u\otimes \tau\|_{H^2} + \||\nabla|^{-1} u\|_{H^2}\|Q\|_{H^2}\\
\lesssim& \| u\|_{H^2}^2 + \| u\|_{H^2}^2\|\tau\|_{H^2}\\
&+ \||\nabla|^{-1} u\|_{H^2}\big(\|\nabla u\|_{L^\infty}\|\tau\|_{H^2} + \|\nabla^2 u\|_{L^6}\|\tau\|_{W^{1,3}} + \|\nabla^3 u\|_{L^2}\| \tau\|_{L^\infty}\big)\\
\lesssim& \| u\|_{H^2}^2 + \| u\|_{H^2}^2\|\tau\|_{H^2}+ \||\nabla|^{-1} u\|_{H^2}\|\nabla^2 u\|_{H^1}\|\tau\|_{H^2}.
\end{split}
\end{equation}
Thus, we get
\begin{equation}\label{eqN7}
\begin{split}
\Big | \int_{0}^{t} N_7(t') dt' \Big | \lesssim& \sup_{0\leq t' \leq t} \||\nabla|^{-1} u(t')\|_{H^2} \|\tau(t')\|_{H^2} + \int_0^t \| u\|_{H^2}^2  dt'\\
 &+\sup_{0 \leq t' \leq t}\|\tau(t')\|_{H^2} \int_0^t  \| u\|_{H^2}^2 dt' \\
 &+ \sup_{0\leq t' \leq t} \||\nabla|^{-1} u(t')\|_{H^2}\|\tau(t')\|_{H^2} \int_0^t \|\nabla^2 u\|_{H^1} dt'\\
 \lesssim& \mathcal{E}_{0,1}(t) + \mathcal{E}_0^\frac{3}{2}(t) + \mathcal{E}_0(t)\mathcal{E}_2^\frac{1}{2}(t)\\
\lesssim& \mathcal{E}_{0,1}(t) + \mathcal{E}_0^\frac{3}{2}(t) + \mathcal{E}_2^\frac{3}{2}(t).
\end{split}
\end{equation}

Integrating \eqref{eq:2.3} with time, according to the estimates \eqref{eqN5}, \eqref{eqN6}, \eqref{eqN7} and Young inequality, we can get the estimate of $\mathcal{E}_{0,2}(t)$ as follows
\begin{equation}\label{eqE02}
\begin{split}
\mathcal{E}_{0,2}(t) =&  \int_{0}^t \||\nabla|^{-1}\mathbb{P}\nabla \cdot \tau \|_{H^2}^2 dt'\\
\lesssim& \mathcal{E}_{0,1}(t) + \mathcal{E}_0^\frac{3}{2}(t) + \mathcal{E}_2^\frac{3}{2}(t).
\end{split}
\end{equation}
We now combine the estimates of $\mathcal{E}_{0,1}(t)$ and $\mathcal{E}_{0,2}(t)$ together to finish this lemma's proof.
Multiplying \eqref{eqE01} by a suitable large number and plus \eqref{eqE02}, we finally obtain
\begin{equation}\nonumber
\mathcal{E}_0(t) \lesssim \mathcal{E}_0(0) + \mathcal{E}_0^\frac{3}{2}(t) + \mathcal{E}_2^\frac{3}{2}(t).
\end{equation}

\end{proof}

Next, we shall consider the time-weighted energy $\mathcal{E}_2(t)$ which represents the good decay properties of higher order norms of solutions and then give the following lemma.
~\\~\\~
\begin{lemma}\label{lem2}
The energies are defined in \eqref{energy1} and \eqref{energy2}, then we have
\begin{equation}\nonumber
\mathcal{E}_2(t) \lesssim \mathcal{E}_0(t) + \mathcal{E}_0^\frac{3}{2}(t) + \mathcal{E}_2^\frac{3}{2}(t).
\end{equation}
\end{lemma}

\begin{proof}
Like the process in Lemma \ref{lem1}, we first divide the proof into two parts. Define $\mathcal{E}_{2,1}$ and $\mathcal{E}_{2,2}$ as follows
\begin{align*}
\mathcal{E}_{2,1}(t) =& \sup_{0 \leq t' \leq t} (1+t')^2(\|\nabla u(t')\|_{H^1}^2 + 2\| \mathbb{P}\nabla \cdot \tau(t')\|_{H^1}^2) + \int_{0}^{t} (1+t')^2\|\nabla^2 u(t')\|_{H^1}^2 \;dt', \\
\mathcal{E}_{2,2}(t) =& \int_{0}^{t} (1+t')^2  \|\nabla \mathbb{P}\nabla \cdot \tau \|_{L^2}^2 \; dt' .
\end{align*}
Then $\mathcal{E}_2 = \mathcal{E}_{2,1} + \mathcal{E}_{2,2}$, we shall first give the estimate of $\mathcal{E}_{2,1}$ .
~\\~\\~
\textbf{First Step:}
~\\~\\~
Operating $\nabla^{k+1} (k = 0, 1)$ derivative on the first equation of system \eqref{eq:1.1} and then operating $\nabla^k \mathbb{P} \nabla \cdot $ on the second equation of system \eqref{eq:1.1}, we will get the following system
\begin{equation}\label{eq:2.4}
\begin{cases}
\nabla^{k+1}u_t + \nabla^{k+1}(u\cdot \nabla u) -  \nabla^{k+1}\Delta u + \nabla^{k+1}\nabla p =  \nabla^{k+1}\nabla \cdot \tau, \\
\nabla^k \mathbb{P} \nabla \cdot\tau_t + \nabla^k \mathbb{P} \nabla \cdot(u\cdot \nabla \tau)  + \nabla^k \mathbb{P} \nabla \cdot Q(\tau, \nabla u) =  \frac{1}{2}\nabla^k  \Delta u .
\end{cases}
\end{equation}
Notice the coefficients in above system, we take inner product with $\nabla^{k+1} u$ for the first equation of \eqref{eq:2.4} and take inner product with $2\nabla^k \mathbb{P} \nabla \cdot \tau$ for the second equation of \eqref{eq:2.4}. Adding the time weight
$(1+t)^2$, we will get
\begin{equation}\label{eq:2.5}
\frac{1}{2}\frac{d}{dt}(1+t)^2 \big(\|\nabla u(t)\|_{H^1}^2 + 2\| \mathbb{P}\nabla \cdot \tau(t)\|_{H^1}^2 \big)
+ (1+t)^2\|\nabla^2 u(t)\|_{H^1}^2 = M_1 + M_2 + M_3 + M_4 + M_5,
\end{equation}
where,
\begin{equation}\nonumber
\begin{split}
M_1 =& (1+t)^2\sum_{k = 0}^1  \int \nabla^{k+1}\nabla \cdot \tau \nabla^{k+1} u + \nabla^k  \Delta u \nabla^k \mathbb{P} \nabla \cdot \tau \; dx ,\\
M_2 =& (1+t)\big(\|\nabla u(t)\|_{H^1}^2 + 2\| \mathbb{P}\nabla \cdot \tau(t)\|_{H^1}^2 \big),\\
M_3 =& - (1+t)^2\sum_{k = 0}^1 \int \nabla^{k+1}(u\cdot \nabla u) \nabla^{k+1} u \;dx,\\
M_4 =& - 2(1+t)^2\sum_{k = 0}^1 \int \nabla^k \mathbb{P} \nabla \cdot(u\cdot \nabla \tau) \nabla^k \mathbb{P} \nabla \cdot \tau \;dx,\\
M_5 =& - 2(1+t)^2\sum_{k = 0}^1 \int \nabla^k \mathbb{P} \nabla \cdot Q(\tau, \nabla u) \nabla^k \mathbb{P} \nabla \cdot \tau \;dx.
\end{split}
\end{equation}

Similarly, we shall estimate each term on the right hand side of \eqref{eq:2.5}.
First, for the term $M_1$, using integration by parts and divergence free condition $\nabla \cdot u = 0$, we can compute
\begin{equation}\label{eqM1}
\begin{split}
M_1 =& (1+t)^2\sum_{k = 0}^1  \int -\nabla^k \nabla \cdot \tau \nabla^k \Delta u + \nabla^k  \Delta \mathbb{P} u \nabla^k \nabla \cdot \tau \;dx\\
=& (1+t)^2\sum_{k = 0}^1  \int -\nabla^k \nabla \cdot \tau \nabla^k \Delta u + \nabla^k  \Delta  u \nabla^k \nabla \cdot \tau \;dx\\
=& 0.
\end{split}
\end{equation}

For the term $M_2$, we can directly derive
\begin{equation}\label{eqM2}
\begin{split}
\int_{0}^t |M_2(t')| dt' \lesssim& \int_{0}^t (1+t')\big(\|\nabla u(t')\|_{H^1}^2 + 2\| \mathbb{P}\nabla \cdot \tau(t')\|_{H^1}^2\big) \;dt'\\
\lesssim& \mathcal{E}_1(t)\\
\lesssim& \mathcal{E}_0^\frac{1}{2}(t) \mathcal{E}_2^\frac{1}{2}(t).
\end{split}
\end{equation}

Using integration by parts, H\"{o}lder inequality and Sobolev imbedding theorem, we get the estimate of $M_3$ as follows
\begin{equation}\nonumber
\begin{split}
|M_3| \lesssim& (1+t)^2\|u \cdot \nabla u\|_{H^1} \|\nabla^2 u\|_{H^1}\\
\lesssim& (1+t)^2\|u\|_{W^{1,3}}\|\nabla u\|_{W^{1,6}}\|\nabla^2 u\|_{H^1}\\
\lesssim& (1+t)^2\|u\|_{H^2} \|\nabla^2 u\|_{H^1} ^2.
\end{split}
\end{equation}
Hence,
\begin{equation}\label{eqM3}
\begin{split}
\int_0^t |M_3(t')| dt' \lesssim& \sup_{0 \leq t' \leq t} \|u(t')\|_{H^2}\int_0^t (1+t')^2 \|\nabla^2 u\|_{H^1} ^2 \;dt'\\
\lesssim& \mathcal{E}_0^\frac{1}{2}(t)\mathcal{E}_2(t).
\end{split}
\end{equation}

Now, we turn to the wildest term $M_4$. Our strategy is to apply Proposition \ref{prop} and divide $M_4$ into three more achievable parts. We have $M_4 = M_{4,1}+M_{4,2}+M_{4,3}$, where,
\begin{equation}\nonumber
\begin{split}
M_{4,1} =& - 2(1+t)^2\sum_{k = 0}^1\int \nabla^k \mathbb{P} (u\cdot \nabla \mathbb{P}\nabla \cdot \tau) \nabla^k \mathbb{P} \nabla \cdot \tau dx,\\
M_{4,2} =& - 2(1+t)^2\sum_{k = 0}^1\int \nabla^k \mathbb{P} (\nabla u\cdot \nabla \tau) \nabla^k \mathbb{P} \nabla \cdot \tau dx,\\
M_{4,3} =& 2(1+t)^2\sum_{k = 0}^1\int \nabla^k \mathbb{P}\big(\nabla u\cdot \nabla \Delta^{-1} \nabla \cdot \nabla \cdot \tau \big) \nabla^k \mathbb{P} \nabla \cdot \tau dx.
\end{split}
\end{equation}
Notice the fact $\mathbb{P} \mathbb{P} = \mathbb{P}$, using integration by parts and divergence free condition,
we can compute $M_{4,1}$ like follows
\begin{equation}\nonumber
\begin{split}
M_{4,1} =& - 2(1+t)^2\sum_{k = 0}^1\int \nabla^k (u\cdot \nabla \mathbb{P}\nabla \cdot \tau) \nabla^k \mathbb{P}\mathbb{P} \nabla \cdot \tau \;dx \\
=& - 2(1+t)^2\sum_{k = 0}^1\int \nabla^k (u\cdot \nabla \mathbb{P}\nabla \cdot \tau) \nabla^k \mathbb{P} \nabla \cdot \tau \; dx\\
=& - 2(1+t)^2\int  \nabla u\cdot \nabla \mathbb{P}\nabla \cdot \tau \; \nabla \mathbb{P} \nabla \cdot \tau \;dx.
\end{split}
\end{equation}
Then, applying H\"{o}lder inequality and Sobolev imbedding theorem, we derive
\begin{equation}\nonumber
\begin{split}
|M_{4,1} | \lesssim& (1+t)^2\|\nabla u\cdot \nabla \mathbb{P}\nabla \cdot \tau\|_{L^2} \|\nabla \mathbb{P} \nabla \cdot \tau\|_{L^2}\\
\lesssim& (1+t)^2\|\nabla u\|_{L^\infty} \|\nabla \mathbb{P} \nabla \cdot \tau\|_{L^2}^2\\
\lesssim& (1+t)^2\|\nabla^2 u\|_{H^1} \|\nabla \mathbb{P} \nabla \cdot \tau\|_{L^2}^2 .
\end{split}
\end{equation}
For the estimate of $M_{4,2}$, we shall use the property that Riesz operator $\mathcal{R}_i = (-\Delta)^{-\frac{1}{2}} \nabla_i$ is $L^r$ bounded for $1< r < \infty$. Hence, it also holds for any suitable regular vector $v$, $ \|\mathbb{P}v \|_{L^r} \lesssim \|v \|_{L^r}$. Using divergence free condition, H\"{o}lder inequality and Sobolev imbedding theorem we get
\begin{equation}\nonumber
\begin{split}
|M_{4,2}| \lesssim & (1+t)^2\big(\|\nabla u \; \tau\|_{L^2} \| \nabla \mathbb{P} \nabla \cdot \tau\|_{L^2} + \|\nabla(\nabla u \cdot \nabla \tau)\|_{L^2} \|\nabla \mathbb{P} \nabla \cdot \tau \|_{L^2}\big)\\
\lesssim& (1+t)^2 \|\tau\|_{H^1}\|\nabla^2 u\|_{L^2}\|\nabla \mathbb{P} \nabla \cdot \tau\|_{L^2}\\
&+(1+t)^2 \big(\|\nabla u\|_{L^\infty}\|\nabla^2 \tau\|_{L^2}+ \|\nabla^2 u\|_{L^6}\|\nabla \tau\|_{L^3} \big)\|\nabla \mathbb{P} \nabla \cdot \tau \|_{L^2}\\
\lesssim& (1+t)^2 \|\tau\|_{H^2} \|\nabla^2 u\|_{H^1} \|\nabla \mathbb{P} \nabla \cdot \tau \|_{L^2}.
\end{split}
\end{equation}
We can use the same method in the estimate of $M_{4,3}$,
\begin{equation}\nonumber
|M_{4,3}| \lesssim (1+t)^2 \|\tau\|_{H^2} \|\nabla^2 u\|_{H^1} \|\nabla \mathbb{P} \nabla \cdot \tau \|_{L^2}.
\end{equation}
Hence,
\begin{equation}\label{eqM4}
\begin{split}
\int_0^t |M_4(t')| dt' \lesssim& \sup_{0 \leq t'\leq t}  \|\tau(t')\|_{H^2} \int_0^t (1+t')^2\|\nabla^2 u\|_{H^1} \|\nabla \mathbb{P} \nabla \cdot \tau \|_{L^2} \;dt'\\
\lesssim& \mathcal{E}_0^\frac{1}{2}(t) \mathcal{E}_2(t).
\end{split}
\end{equation}

For the last term $M_5$, using integration by parts, H\"{o}lder inequality and Sobolev imbedding theorem, we have
\begin{equation}\nonumber
\begin{split}
|M_5| \lesssim& (1+t)^2\big( \|Q\|_{L^2}\|\nabla \mathbb{P}\nabla \cdot \tau\|_{L^2}+ \|\nabla^2 Q\|_{L^2}\|\nabla \mathbb{P}\nabla \cdot \tau\|_{L^2}\big)\\
\lesssim& (1+t)^2\big(\|\nabla u\|_{L^\infty}\|\tau\|_{H^2}+ \|\nabla^2 \tau\|_{L^6} \|\nabla \tau\|_{L^3}+ \|\nabla^3 u\|_{L^2}\|\tau\|_{L^\infty}\big)\|\nabla \mathbb{P}\nabla \cdot \tau\|_{L^2}\\
\lesssim& (1+t)^2 \|\tau\|_{H^2}\|\nabla^2 u\|_{H^1}\|\nabla \mathbb{P}\nabla \cdot \tau\|_{L^2}.
\end{split}
\end{equation}
Thus, we get the following
\begin{equation}\label{eqM5}
\begin{split}
\int_0^t |M_5(t')| dt' \lesssim& \sup_{0 \leq t' \leq t} \|\tau(t')\|_{H^2}\int_0^t (1+t')^2 \|\nabla^2 u\|_{H^1}\|\nabla \mathbb{P}\nabla \cdot \tau\|_{L^2} \;dt'\\
\lesssim& \mathcal{E}_0^\frac{1}{2}(t)\mathcal{E}_2(t).
\end{split}
\end{equation}

Integrating \eqref{eq:2.5} with time and applying the estimates of $M_1 \sim M_5$, i.e., \eqref{eqM1}, \eqref{eqM2}, \eqref{eqM3}, \eqref{eqM4} and \eqref{eqM5}, it holds that
\begin{equation}\label{eqE20}
\begin{split}
\mathcal{E}_{2,1}(t) =& \sup_{0 \leq t' \leq t} (1+t')^2 \big(\|\nabla u(t')\|_{H^1}^2 + 2\| \mathbb{P}\nabla \cdot \tau(t')\|_{H^1}^2 \big) + \int_{0}^{t} (1+t')^2 \|\nabla^2 u(t')\|_{H^1}^2 dt'\\
\lesssim& \mathcal{E}_0(t) + \mathcal{E}_0^\frac{1}{2}(t)\mathcal{E}_2^\frac{1}{2}(t) + \mathcal{E}_0^\frac{3}{2}(t) + \mathcal{E}_2^\frac{3}{2}(t).
\end{split}
\end{equation}
~\\~\\~
\textbf{Second Step:}
~\\~\\~
The left work is to do the estimate of $\mathcal{E}_{2,2}(t)$.
Operating $\nabla \mathbb{P}$ on the first equation of system \eqref{eq:1.1}, we have the following equality
\begin{equation}\nonumber
\nabla u_t + \nabla \mathbb{P}(u\cdot \nabla u) - \nabla \Delta u = \nabla \mathbb{P}\nabla \cdot \tau.
\end{equation}
Then, taking $L^2$ inner product with $\nabla \mathbb{P} \nabla \cdot \tau$, we get
\begin{equation}\label{eq:2.6}
(1+t)^2\|\nabla \mathbb{P}\nabla \cdot \tau\|_{L^2}^2 = M_6 + M_7 + M_8,
\end{equation}
where,
\begin{equation}\nonumber
\begin{split}
M_6 =& -(1+t)^2 \int \nabla\Delta u \nabla\mathbb{P}\nabla \cdot \tau \;dx,\\
M_7=& (1+t)^2 \int \nabla\mathbb{P}(u\cdot \nabla u)\nabla\mathbb{P}\nabla \cdot \tau \;dx,\\
M_8=& (1+t)^2  \int \nabla u_t \nabla\mathbb{P}\nabla \cdot \tau \; dx.
\end{split}
\end{equation}

For the first term $M_6$, by H\"{o}lder inequality we directly know that
\begin{equation}\nonumber
|M_6| \lesssim (1+t)^2\|\nabla^3 u\|_{L^2} \|\nabla \mathbb{P} \nabla \cdot \tau\|_{L^2}.
\end{equation}
Thus we have the following estimate
\begin{equation}\label{eqM6}
\int_0^t |M_6(t')| dt' \lesssim \mathcal{E}_{2,1}^\frac{1}{2}(t) \mathcal{E}_{2,2}^\frac{1}{2}(t).
\end{equation}

Obviously, we can get
\begin{equation}\nonumber
\begin{split}
|M_7| \lesssim& (1+t)^2 \big(\|\nabla u \nabla u\|_{L^2} + \|u\nabla^2 u\|_{L^2}\big)\|\nabla \mathbb{P}\nabla \cdot \tau\|_{L^2}\\
\lesssim& (1+t)^2\big(\|\nabla u\|_{L^\infty}\|\nabla u\|_{L^2}+ \|u\|_{L^\infty}\|\nabla^2 u\|_{L^2} \big)\|\nabla \mathbb{P}\nabla \cdot \tau\|_{L^2}\\
\lesssim& (1+t)^2\|u\|_{H^2} \|\nabla^2 u\|_{H^1}\|\nabla \mathbb{P}\nabla \cdot \tau\|_{L^2}.
\end{split}
\end{equation}
Thus,
\begin{equation}\label{eqM7}
\begin{split}
\int_0^t |M_7(t')| dt' \lesssim& \sup_{0 \leq t' \leq t}\|u(t')\|_{H^2} \int_{0}^{t} (1+t')^2 \|\nabla^2 u\|_{H^1}\|\nabla \mathbb{P}\nabla \cdot \tau\|_{L^2} dt'\\
\lesssim& \mathcal{E}_0^\frac{1}{2}(t) \mathcal{E}_2(t).
\end{split}
\end{equation}

Now, we turn to the last term $M_8$. Using integration by parts, we first rewrite this term into the following form
\begin{equation}\nonumber
\begin{split}
M_8 =& \frac{d}{dt} (1+t)^2 \int \nabla u \; \nabla \mathbb{P}\nabla \cdot \tau \;dx
-2(1+t) \int \nabla u \;\nabla\mathbb{P}\nabla \cdot \tau \;dx\\
&-(1+t)^2  \int \nabla u \;\nabla\mathbb{P}\nabla \cdot \tau_t \;dx
\end{split}
\end{equation}
Applying \eqref{eqq} to the last part in $M_8$, it then becomes
\begin{align*}
& -(1+t)^2  \int \nabla \; u \nabla\mathbb{P}\nabla \cdot \tau_t \;dx \\
= &-(1+t)^2  \int \nabla u \; \nabla \Big[\frac{1}{2}\Delta u - \mathbb{P}\nabla \cdot (u\cdot \nabla \tau) + \mathbb{P} \nabla \cdot Q(\tau,\nabla u)\Big] \;dx.
\end{align*}
Hence, using integration by parts, H\"{o}lder inequality and Sobolev imbedding theorem, we have the following estimates
\begin{equation}\nonumber
\begin{split}
&(1+t)^2 \Big| \int \nabla u \; \nabla\mathbb{P}\nabla \cdot \tau_t \;dx \Big|\\
\lesssim& (1+t)^2\|\nabla^2 u\|_{L^2}\big(\|\nabla^2 u\|_{L^2}+\|\mathbb{P}\nabla\cdot (u\cdot \nabla \tau)\|_{L^2}+\|\nabla \cdot Q\|_{L^2}\big)\\
\lesssim& (1+t)^2\|\nabla^2 u\|_{L^2}\big(\|\nabla^2 u\|_{L^2}+\|\mathbb{P}\nabla\cdot (u\cdot \nabla \tau)\|_{L^2}+
\|\nabla u\|_{L^\infty}\|\nabla \tau\|_{L^2}+\|\nabla^2 u\|_{L^2}\|\tau\|_{L^\infty}\big)\\
\lesssim& (1+t)^2\|\nabla^2 u\|_{L^2}\big(\|\nabla^2 u\|_{L^2}+\|\mathbb{P}\nabla\cdot (u\cdot \nabla \tau)\|_{L^2}+
\|\nabla^2 u\|_{H^1}\| \tau\|_{H^2}\big).
\end{split}
\end{equation}
Using the same strategy in the estimate of $M_4$, we apply Proposition \ref{prop} to $\|\mathbb{P}\nabla\cdot (u\cdot \nabla \tau)\|_{L^2}$ and get the following
\begin{equation}\nonumber
\begin{split}
\|\mathbb{P}\nabla\cdot (u\cdot \nabla \tau)\|_{L^2} \lesssim &
\|u\cdot \nabla \mathbb{P}\nabla \cdot \tau\|_{L^2}+\|\nabla u\cdot \nabla \tau\|_{L^2}
+\|\nabla u\cdot \nabla \Delta^{-1} \nabla \cdot \nabla \cdot \tau\|_{L^2}\\
\lesssim& \|u\|_{L^\infty}\|\nabla \mathbb{P}\nabla \cdot \tau\|_{L^2}+\|\nabla u\|_{L^\infty} \|\nabla \tau\|_{L^2}\\
\lesssim& \|u\|_{H^2}\|\nabla \mathbb{P}\nabla \cdot \tau\|_{L^2}+\|\nabla \tau\|_{L^2}\|\nabla^2 u\|_{H^1}.
\end{split}
\end{equation}
Thus,
\begin{equation}\nonumber
\begin{split}
&(1+t)^2 \Big| \int \nabla u \; \nabla \mathbb{P}\nabla \cdot \tau_t \;dx \Big|\\
\lesssim& (1+t)^2\|\nabla^2 u\|_{L^2}\big(\|\nabla^2 u\|_{L^2}+\|u\|_{H^2}\|\nabla \mathbb{P}\nabla \cdot \tau\|_{L^2}+
\|\nabla^2 u\|_{H^1}\| \tau\|_{H^2}\big).
\end{split}
\end{equation}
And then we can get the estimate of $M_8$
\begin{equation}\label{eqM8}
\begin{split}
\Big|\int_0^t (1+t')^2 M_8(t') \;dt'\Big| \lesssim & \sup_{0\leq t'\leq t} (1+t')^2 \|\nabla u\|_{L^2}\|\nabla \mathbb{P}\nabla \cdot \tau\|_{L^2}\\
&+ \int_0^t (1+t') \|\nabla u\|_{L^2} \|\nabla \mathbb{P}\nabla \cdot \tau\|_{L^2} \;dt'\\
&+\int_0^t (1+t')^2 \|\nabla^2 u\|_{L^2}^2\; dt'\\
 &+ \sup_{0 \leq t'\leq t} \|u(t')\|_{H^2}\int_0^{t} (1+t')^2 \|\nabla^2 u\|_{L^2}\|\nabla \mathbb{P}\nabla \cdot \tau\|_{L^2}\;dt'\\
 &+\sup_{0 \leq t'\leq t} \|\tau(t')\|_{H^2}\int_0^{t} (1+t')^2 \|\nabla^2 u\|_{H^1}^2 \;dt'\\
 \lesssim& \mathcal{E}_{2,1}(t) + \mathcal{E}_0^\frac{1}{2}(t)\mathcal{E}_2^\frac{1}{2} + \mathcal{E}_0^\frac{1}{2}\mathcal{E}_2(t).
\end{split}
\end{equation}

Integrating \eqref{eq:2.6} with time, applying the estimates of $M_6 \sim M_8$, i.e., \eqref{eqM6}, \eqref{eqM7}, \eqref{eqM8} and Young inequality, we finally get
\begin{equation}\label{eqE21}
\mathcal{E}_{2,2}(t)
\lesssim  \mathcal{E}_{2,1}(t) + \mathcal{E}_0(t) + \mathcal{E}_0^\frac{1}{2}(t) \mathcal{E}_2^\frac{1}{2}(t) + \mathcal{E}_0^\frac{1}{2}\mathcal{E}_2(t).
\end{equation}
We now combine the estimates of $\mathcal{E}_{2,1}(t)$ and $\mathcal{E}_{2,2}(t)$ together to finish this lemma's proof.
Multiplying \eqref{eqE20} by a suitable large number and plus \eqref{eqE21}, we get
\begin{equation}\nonumber
\mathcal{E}_2(t) \lesssim \mathcal{E}_0(t) + \mathcal{E}_0^\frac{1}{2}(t) \mathcal{E}_2^\frac{1}{2}(t) + \mathcal{E}_0^\frac{3}{2}(t) + \mathcal{E}_2^\frac{3}{2}(t).
\end{equation}
We complete the proof of this lemma by applying Young inequality on the above inequality.
\end{proof}

\section{Proof of the Theorem \ref{thm}}
In this section, we will combine the above \textit{a priori} estimates of $\mathcal{E}_0$ and $\mathcal{E}_2$ together and then give the proof of Theorem \ref{thm}.
First, we define the total energy $\mathcal{E}(t) = \mathcal{E}_0(t) + \mathcal{E}_2(t)$. Notice the estimates in Lemma \ref{lem1} and Lemma \ref{lem2}, we have
\begin{equation}\label{eq:4.1}
\mathcal{E}(t) \leq C_1 \mathcal{E}_0(0) + C_1\mathcal{E}^\frac{3}{2}(t),
\end{equation}
 for some positive constant $C_1$.  Under the setting of initial data in Theorem \ref{thm}, there exists a small enough number $\varepsilon$ such that $\mathcal{E}(0), C_1\mathcal{E}_0(0) \leq \varepsilon$. Due to local existence theory which can be achieved through standard energy method (see \cite{CM} for instance), there exists a positive time $T$ such that
\begin{equation}\label{eqEtotal2}
\mathcal{ E}(t) \leq 2 \varepsilon , \quad  \forall \; t \in [0, T].
\end{equation}
Let $T^{*}$ be the largest possible time of $T$ for what \eqref{eqEtotal2} holds. Now, we only need to show $T^{*} = \infty$. By the estimate of total energy \eqref{eq:4.1}, we can use
 a standard continuation argument to show  $T^{*} = \infty$ provided that $\varepsilon$ is small enough.  We omit the details here. Hence, we finish the proof of Theorem \ref{thm}.

\section*{Acknowledgement}

The author sincerely appreciates the helpful suggestion from Professor Yi Zhou and Professor Ting Zhang.

\end{document}